\documentclass[12pt,a4paper]{amsart}
\calclayout
\usepackage{hyperref}
\usepackage{comment}
\usepackage{mathtools}
\usepackage{autonum}
\DeclarePairedDelimiter\ceil{\lceil}{\rceil}
\DeclarePairedDelimiter\floor{\lfloor}{\rfloor}

\newtheorem{thm}{Theorem}[section]
\newtheorem{lem}[thm]{Lemma}
\newtheorem{prop}[thm]{Proposition}
\numberwithin{equation}{section}

\title[A 2-spine decomposition and Yaglom's theorem]
{\large A 2-spine Decomposition of the Critical Galton-Watson Tree and a Probabilistic Proof of Yaglom's Theorem}
\author{Yan-Xia Ren, Renming Song and Zhenyao Sun}
\address{
	Yan-Xia Ren\\
	School of Mathematical Sciences \& Center for Statistical Science\\
	Peking University, Beijing\\
	P. R. China, 100871}
\email{yxren@math.pku.edu.cn}
\thanks{The research of Yan-Xia Ren is supported in part by NSFC (Grant Nos. 11671017 and 11731009)}
\address{
	Renming Song\\
	Dept of Mathematics\\
	University of Illinois at Urbana-Champaign\\
	Urbana, IL 61801}
\email{rsong@illinois.edu}
\thanks{The research of Renming Song is supported in part by the Simons Foundation (\#429343, Renming Song)}
\address{
	Zhenyao Sun\\
	School of Mathematical Sciences\\
	Peking University\\
	Beijing, P. R. China, 100871}
\email{zhenyao.sun@pku.edu.cn}
\thanks{Zhenyao Sun is supported by the China Scholarship Council. Corresponding author.}
\begin{document}
\begin{abstract}
	In this note  we propose a two-spine decomposition of the critical Galton-Watson tree and use this decomposition to give a probabilistic proof of Yaglom's theorem.
\end{abstract}
\maketitle	
\section{Introduction}
\subsection{Model}
\label{sec:model}
	Consider a critical Galton-Watson process
	$(Z_n)_{n\ge 0}$ 	with $Z_0 = 1$
	and offspring distribution $\mu$ on $\mathbb N_0 : = \{0,1,\dots\}$ which has mean $1$ and finite variance $\sigma^2>0$, i.e.,
\begin{equation}\label{eq:mean}
 \sum_{k=0}^\infty k \mu(k)	=1
\end{equation}
	and
\begin{equation}\label{eq:variance}
	0	
	<	\sigma^2
	:=	\sum_{k=0}^\infty  (k-1)^2 \mu(k)
	=	\sum_{k=0}^\infty k(k-1) \mu(k)
	<	\infty.
\end{equation}
	For simplicity,  
	we will refer to $(Z_n)_{n\geq 0}$ as a  \emph{$\mu$-Galton-Watson process}.
	It is well known that
\begin{thm}[\cite{kesten1966galton}] \label{thm: Kolmogrov and Yaglom theorem}
	For a $\mu$-Galton-Watson process $(Z_n)_{n\geq 0}$
	satisfying \eqref{eq:mean} and \eqref{eq:variance}, we have
\begin{enumerate}
\item \label{thm:kolmogorov}
	 $n P (Z_n>0) \xrightarrow[n \to \infty]{} 2/\sigma^2;$
\item \label{thm:yaglom}
	$\{n^{-1}Z_n; P(\cdot | Z_n>0)\}\xrightarrow[n \to \infty]{d} Y,$
\end{enumerate}
	where $Y$ is an exponential random variable with mean $\sigma^2/2$.
\end{thm}

	Under a third moment assumption, assertions \eqref{thm:kolmogorov} and \eqref{thm:yaglom} of Theorem \ref{thm: Kolmogrov and Yaglom theorem} are due to \cite{kolmogorov1938losung} and \cite{yaglom1947certain} respectively.
	Theorem \ref{thm: Kolmogrov and Yaglom theorem}(2) is usually called Yaglom's theorem.
	For probabilistic proofs of the above results, we refer our readers to
\cite{geiger1999elementary}, \cite{geiger2000new} and \cite{lyons1995conceptual}.

	In \cite{lyons1995conceptual}, Lyons, Pemantle and Peres gave a probabilistic proof of Theorem \ref{thm: Kolmogrov and Yaglom theorem} using the so-called size-biased $\mu$-Galton-Watson tree.
	In this note, by \emph{size-biased transform} we mean the following:
	Let $X$ be a random variable
	and $g(X)$ be a Borel function of $X$ with $P(g(X) \geq 0) = 1$ and $E[g(X)]\in (0,\infty)$.
	We say a random variable $W$ is
	a $g(X)$-size-biased transform (or simply $g(X)$-transform) of $X$ if
\[
	E[f( W )] 
	= \frac{ E[g(X)f(X)]}{E[g(X)]}
\]
	for each positive Borel function $f$.
	An $X$-transform of $X$ is sometimes called a size-biased transform of $X$.

	We now recall the size-biased $\mu$-Galton-Watson tree introduced in \cite{lyons1995conceptual}.
	Let $L$ be a random variable with distribution $\mu$.
   Denote by $\dot L$ an \emph{$L$-transform} of $L$.
	The celebrated \emph{size-biased $\mu$-Galton-Watson tree} is then constructed as follows:
\begin{itemize}
\item
	There is an initial particle which is marked.
\item
	Any marked particle gives independent birth to a random number of children according to $\dot L$. Pick one of those children randomly as the new marked particle while leaving the other children as unmarked particles.
\item
	Any unmarked particle gives 
birth independently to a random number of unmarked children according to $L$.
\item
	The evolution goes on.
\end{itemize}

	Notice that the marked particles form a descending family line which will be referred to as the \emph{spine}.
	Define $\dot Z_n$ as the population of the $n$th generation in the size-biased tree.
	It is proved in \cite{lyons1995conceptual} that the process $(\dot Z_n)_{n\ge 0}$ is a martingale transform of the process $(Z_n)_{n\ge 0}$ via the martingale $(Z_n)_{n\ge 0}.$
	That is, for any generation number $n$ and any bounded Borel function $g$ on $\mathbb N_0^{n}$,
\begin{equation}
\label{eq:htransformation}
	E [ g ( \dot Z_1, \dots, \dot Z_n) ]
	= \frac { E[ Z_n g( Z_1, \dots, Z_n)]} {E [ Z_n]}.
\end{equation}

	It is natural to consider probabilistic proofs of analogous results of Theorem \ref{thm: Kolmogrov and Yaglom theorem} for more general critical branching processes.
	Vatutin and  Dyakonova \cite{VD} gave a probabilistic proof of Theorem \ref{thm: Kolmogrov and Yaglom theorem}(1) for multitype critical branching processes.
	As far as we know, there is no probabilistic proof of Yaglom's theorem for multitype critical branching processes.
	It seems that it is difficult to adapt the probabilistic proofs in \cite{geiger2000new} and \cite{lyons1995conceptual} for monotype branching processes to more general models, such as multitype branching processes, branching Hunt processes and superprocesses.

	In this note, we propose a $k(k-1)$-type size-biased $\mu$-Galton-Watson tree equipped with a two-spine skeleton, which serves as a change-of-measure of the original $\mu$-Galton-Watson tree;
	and with the help of this two-spine technique, we give a new probabilistic proof of Theorem \ref{thm: Kolmogrov and Yaglom theorem}(2), i.e. Yaglom's theorem.
	The main motivation for developing this new proof for the classical Yaglom's theorem is that this new method is generic, in the sense that it can be generalized to more complicated critical branching systems.
	In fact, in our	follow-up
	paper \cite{RenSongSun2017Spine}, we show that, in a similar spirit, a two-spine structure can be constructed for a class of critical superprocesses, and a probabilistic proof of a Yaglom type theorem can be obtained for those processes.

	Another aspect of our new proof is that we take advantage of a fact that the exponential distribution can be characterized by a particular $x^2$-type size-biased distributional equation.
	An intuitive explanation of our method,
and a comparison with the methods of \cite{geiger2000new} and \cite{lyons1995conceptual}, are
	made in the next subsection.
	We think this new point of view of convergence to the exponential law provides an alternative insight on the classical Yaglom's theorem.

	We now give a formal construction of our $k(k-1)$-type size-biased $\mu$-Galton-Watson tree.
	Denote by $\ddot L$ an \emph{$L(L-1)$-transform} of $L$.
	Fix a generation number $n$ and pick a random generation number $K_n$ uniformly among $\{0,\dots,n-1\}$.
	The \emph{$k(k-1)$-type size-biased $\mu$-Galton-Watson tree with height $n$} is then defined as a particle system such that:
\begin{itemize}
\item
	There is an initial particle which is marked.
\item
	Before or after generation $K_n$, any marked particle gives birth independently to a random number of children according to $\dot L$.
	Pick one of those children randomly as the new marked particle while leaving the other children as unmarked particles.
\item
	The marked particle at generation $K_n$, however, gives birth, independent of other particles, to a random number of children according to $\ddot L$.
	Pick two different particles randomly among those children as the new marked particles while leaving the other children as unmarked particles.
\item
	Any unmarked particle gives birth independently to a random number of unmarked children according to $L$.
\item
	The system stops at generation $n$.
\end{itemize}

	If we track all the marked particles, it is clear that they form a \emph{two-spine skeleton} with $K_n$ being the last generation where those two spines are together.
	It would be helpful to consider this skeleton as two disjoint spines,
	where \emph{the longer spine} is a family line from generation $0$ to $n$ and \emph{the shorter spine} is a family line from generation $K_n+1$ to $n$.
	
	For any $0\le m \le n$, denote by $\ddot Z_m^{(n)}$ the population of the $m$th generation in the $k(k-1)$-type size-biased $\mu$-Galton-Watson tree with height $n$.
	The main reason for proposing such a model is that the process $(\ddot Z_m^{(n)})_{0\le m\le n}$ can be viewed as
    a $Z_n(Z_n-1)$-transform of the process $(Z_m)_{0\le m\le n}$.
	This is made precise in the result below which will be proved in Section \ref{sec:spacesandmeasures}.
\begin{thm}
\label{thm: change of measure}
	Let $(Z_m)_{m\ge 0}$ be a $\mu$-Galton-Watson process and $(\ddot Z_m^{(n)})_{0\le m\le n}$ be the population of a $k(k-1)$-type size-biased $\mu$-Galton-Watson tree with height $n$.
	Suppose that $\mu$ satisfies \eqref{eq:mean} and \eqref{eq:variance}.
	Then, for any bounded Borel function $g$ on $\mathbb N^{n}_0$,
\[
		E[ g ( \ddot Z_1^{(n)}, \dots, \ddot Z_n^{(n)})]
	=
		\frac{ E[ Z_n(Z_n-1) g( Z_1, \dots, Z_n)]} {E [ Z_n ( Z_n - 1)]}.		
\]
\end{thm}

	The idea of considering a branching particle system with more than one spine is not new.
	A particle system with $k$ spines  was constructed in \cite{harris2015many} and used in the  many-to-few formula for branching Markov processes and branching random walks.
	Inspired by \cite{harris2015many}, we use a two-spine model to characterize the $k(k-1)$-type size-biased branching process.

\subsection{Methods.}
\label{sec: Methods}
	Suppose that $X$ is a non-negative 
	random variable with $E[X] \in (0,\infty)$,
	then its distribution conditioned on $\{ X > 0\}$ can be characterized by its conditional expectation $E[X|X>0]$ and its size-biased transform $\dot X$.
	In fact, for each $\lambda \geq 0$,
\begin{equation}
\label{eq: conditional and size-biased transform}
\begin{split}
	&E[1-e^{-\lambda X}|X>0]
	= \frac{E[1-e^{-\lambda X}]}{P(X>0)}
	\\&\quad = \frac{1}{P(X>0)}\int_0^\lambda E[Xe^{-s X}]ds = E[X|X>0]\int_0^\lambda E[e^{-s \dot X}]ds.
\end{split}
\end{equation}
	As a consequence,  
	Theorem \ref{thm: Kolmogrov and Yaglom theorem}	is equivalent to
\begin{equation}
\label{eq: convergence of conditional expectation}
	E\big[\frac{Z_n}{n}| Z_n > 0\big]
	\xrightarrow[n\to \infty]{} \frac{\sigma^2}{2}
\end{equation}
	and
\[
\label{eq: convergence after size-biased}
	E[e^{-s \frac{\dot Z_n}{n}}]
	\xrightarrow[n\to \infty]{} E[e^{-s \dot Y}].
\]
	where $\dot Y$ is a $Y$-transform 
	of the exponential random variable $Y$.
	Indeed, since $E[Z_n] = 1$, \eqref{eq: convergence of conditional expectation} is equivalent to Theorem \ref{thm: Kolmogrov and Yaglom theorem}(\ref{thm:kolmogorov}); 
and assuming \eqref{eq: convergence of conditional expectation}, 
	according to \eqref{eq: conditional and size-biased transform}, we can see \eqref{eq: convergence after size-biased} is equivalent to Theorem \ref{thm: Kolmogrov and Yaglom theorem}(\ref{thm:yaglom}).
	In Section \ref{sec: proofs}, for completeness, we will simplify 
	the argument of \cite{geiger1999elementary} and \cite{VD},
	and give a proof of Theorem \ref{thm: Kolmogrov and Yaglom theorem}(\ref{thm:kolmogorov}).

	Our method of proving
	\eqref{eq: convergence after size-biased}
	takes advantage of a fact that the exponential distribution is characterized by an $x^2$-type size-biased distributional equation.
	This is made precise in the next lemma, which will be proved in Section \ref{sec: proofs}:
\begin{lem} \label{lem: our equation}
	Let $Y$ be a strictly positive random variable with finite second moment.
	Then $Y$ is exponentially distributed if and only if
\begin{equation}
\label{eq: x2 type size-biased equation for exponential distribution}
	\ddot Y \overset{d}
	= \dot Y + U \cdot \dot Y',
\end{equation}
 where $\dot Y$ and $\dot Y'$ are both $Y$-transforms of  $Y$,
$\ddot Y$ is a $Y^2$-transform of $Y$,
	$U$ is a uniform random variable on $[0,1]$, and $\dot Y$, $\dot Y'$, $\ddot Y$
	and $U$ are independent.
\end{lem}	
	With this lemma and Theorem \ref{thm: change of measure}, we can give an intuitive explanation of the exponential convergence in Yaglom's Theorem.
	From the construction of the $k(k-1)$-type size-biased $\mu$-Galton-Watson tree $(\ddot Z^{(n)}_m)_{0\le m\le n}$, we see that the population $\ddot Z^{(n)}_n$ in the $n$th generation can be separated into two parts: descendants
	from the longer spine and descendants from the shorter spine.
	Due to their construction,
	the first part, the descendants from the longer spine at generation $n$,
	is distributed approximately like $\dot Z_n$,
	while the second part, the descendants from the shorter spine at generation $n$, is distributed approximately like $\dot Z_{ \floor{U\cdot n}}$.
	Those two parts are approximately independent of each other.
	So, after a renormalization, we have roughly that
\begin{equation}
\label{eq: Our insight}
	\frac{\ddot Z_n^{(n)}}{n}
	\overset{d} \approx \frac{\dot Z_n}{n} +
	U \cdot \frac{   \dot Z'_{  \floor{ U n }  }   }    {   Un   },
\end{equation}
	where the process $(\dot Z'_m)$ is an independent copy of $(\dot Z_m)$.
	Suppose that $\dot Z_n/n$ converges weakly to a random variable $\dot Y$, and $\ddot Z_n/n$ converges weakly to a random variable $\ddot Y$. 
	Then, according to \cite[Lemma 4.3]{lyons1995conceptual}, $\ddot Y$ is a size-biased transform of $\dot Y$. 
Therefore, letting $n\to\infty$ in \eqref{eq: Our insight}, 
	$\dot Y$ should satisfy \eqref{eq: x2 type size-biased equation for exponential distribution}, which, by Lemma \ref{lem: our equation}, suggests that \eqref{eq: convergence after size-biased} is true.
	
	It is interesting to compare this method of proving exponential convergence with the methods 
	used in \cite{geiger2000new} and \cite{lyons1995conceptual}.
	In \cite{lyons1995conceptual}, Lyons, Pemantle and Peres characterize the exponential distribution by a different
but well-known $x$-type size-biased distributional equation:
	A nonnegative random variable $Y$ with positive finite mean is exponentially distributed if and only if it satisfies that
\begin{equation}
\label{eq: Lyons' distributional equation}
		Y 		\overset{d}= U \cdot \dot Y
\end{equation}
   where $\dot Y$ is a $Y$-transform of $Y$,  and $U$ is a uniform random variable on
$[0,1]$, which is independent of $\dot Y$.
    With the help of the size-biased tree, they then show that $\ceil{U \cdot \dot Z_n}$ is distributed approximately like $Z_n$ conditioned on $\{Z_n > 0\}$.
	So, after a renormalization, they have roughly that
\begin{equation}
\label{eq: Lyons' insight}
	\Big\{\frac{Z_n}{n} ; P(  \cdot| Z_n > 0) \Big\}
	\overset{d}{\approx} U \cdot \frac{ \dot Z_n}{n}.
\end{equation}
	Suppose that $\{Z_n/n; P(\cdot | Z_n > 0)\}$ converge weakly to a random variable $Y$, and $\dot Z_n /n$ converge weakly to a random variable $\dot Y$.
	Then, according to \cite[Lemma 4.3]{lyons1995conceptual}, $\dot Y$ is the size-biased transform of $Y$.
	Therefore, letting $n\to \infty$ in \eqref{eq: Lyons' insight}, 
	$Y$ should satisfy \eqref{eq: Lyons' distributional equation}, which suggests that $Y$ is exponentially distributed.
	
	In \cite{geiger2000new}, Geiger characterizes the exponential distribution by another distributional equation:
	If $Y^{(1)}$ and $Y^{(2)}$ are independent copies of 
 a random variable $Y$ with positive finite variance, 
	and $U$ is an independent uniform random variable on $[0,1]$, then $Y$ is exponentially distributed if and only if
\begin{equation}
\label{eq: Geiger's equation}
	Y	\overset{d} = U (Y^{(1)} + Y^{(2)}).
\end{equation}
	Geiger then shows that for $(Z_n)$, conditioned on non-extinction at generation $n$,
	the distribution of the generation of the most recent common ancestor (MRCA) of the particles at generation $n$ is asymptotically uniform among $\{0,1,\dots,n\}$ (a result due to \cite{Zubkov1975}, see also \cite{geiger1999elementary}), and there are asymptotically two children of
the MRCA, each with at least 1 descendant in generation $n$.
	After a renormalization, roughly speaking, Geiger has that
\begin{equation}
\label{eq: Geiger's insight}
	\Big\{\frac{Z_n}{n} ; P(  \cdot| Z_n > 0) \Big\}
		\overset{d}{\approx} U \cdot \frac{ Z^{(1)}_{  \floor{Un}  }}{Un} + U \cdot \frac{ Z^{(2)}_{ \floor{Un} }}{Un} ,
\end{equation}
	where for each $m$, $Z_m^{(1)}$ and $Z_m^{(2)}$ are independent copies of $\{Z_m; P(\cdot | Z_m > 0)\}$.
	Therefore, if $\{Z_n/n; P(\cdot| Z_n > 0)\}$ converges weakly to a random variable $Y$, then $Y$ should satisfy \eqref{eq: Geiger's equation}, which suggests that $Y$ is exponentially distributed.
	
	From this comparison, we see that all the methods mentioned above share one similarity: They all establish the exponential convergence via some particular distributional equation.
	However, since the equations \eqref{eq: x2 type size-biased equation for exponential distribution}, \eqref{eq: Lyons' distributional equation} and \eqref{eq: Geiger's equation} are different, the actual way of proving the convergence varies.
	In \cite{lyons1995conceptual}, an elegant tightness argument is made along with \eqref{eq: Lyons' insight}.
	However, it seems that this tightness argument is not suitable for \eqref{eq: Geiger's insight}, due to a property that the conditional convergence for some subsequence $Z_{n_k}/n_k$ implies the convergence of $U \cdot \dot Z_{n_k}/n_k$,
   	but does not implies the convergence of $Z^{(i)}_{ \floor{Un_k}}/Un_k, i=1, 2$.
	Instead, a contraction type argument in the $L^2$-Wasserstein metric is used in Geiger \cite{geiger2000new}.
	
	Due to similar reasons, in this note,
	to actually prove the exponential convergence using \eqref{eq: Our insight} and \eqref{eq: x2 type size-biased equation for exponential distribution}, some efforts also must be made.
	We observe that the distributional equation \eqref{eq: Our insight} admits
    a so-called size-biased add-on structure, which is related to L\`evy's
	theory of infinitely divisible distributions: Suppose that $X$ is a nonnegative random variable with $ a := E [X]\in (0,\infty)$,
	then $X$ is infinitely divisible if and only if there exists a nonnegative random variable $A$ independent of $X$ such that $\dot X 	\overset{d} = X + A$.
	In fact,
	the Laplace exponent of $X$ can be expressed as
\[
	-\ln E[ e^{-\lambda X}]
	 =  a \alpha(\{0\}) \lambda+ a \int_{(0,\infty)} \frac{1 - e^{-\lambda y}}{y} \alpha(dy),
\]
	where $\alpha$ is the distribution of $A$.
	Moreover, if $A$ is strictly positive, then
\begin{equation}\label{eq: Laplace exponent for size-biased add-on equation}
	-\ln E[ e^{-\lambda X}]
	=  a  \int_0^\lambda E [e^{-s A}] ds.
\end{equation}
	From this point of view, after considering the Laplace transforms of 
	\eqref{eq: Our insight} and \eqref{eq: x2 type size-biased equation for exponential distribution}, we can establish the convergence of $E[e^{-\lambda \dot Z_n/n}]$ to $E[e^{-\lambda \dot {Y}}]$, which will eventually lead us to Yaglom's theorem.
	This is made precise in Section \ref{sec: proofs}.
	A similar type of argument is also  used in our
	follow-up
	paper \cite{RenSongSun2017Spine} for critical superprocesses.
	
\section{Trees and their decompositions}
\label{sec:preliminary}
\subsection{Spaces and measures}
\label{sec:spacesandmeasures}
	In this subsection, we give a proof of Theorem \ref{thm: change of measure}.
 	Consider \emph{particles} as elements in the space
\[
		\mathcal U
	:=
	\{\emptyset\}\cup\bigcup_{k=1}^\infty \mathbb N^k,
\]
	where $\mathbb N:=\{1,2,\dots\}$.
	Therefore elements in $\mathcal U$ are of the form 213, which we read as the individual being the 3rd child of the 1st child of the 2nd child of the initial ancestor $\emptyset$.
	For two particles $u=u_1\dots u_n, v=v_1\dots v_m\in\mathcal U$, $uv$ denotes the concatenated particle $uv:=u_1\dots u_nv_1\dots v_m$.
	We use the convention $u\emptyset = \emptyset u = u$ and $u_1\dots u_n=\emptyset$ if $n=0$.
	For any particle $u:=u_1\dots u_{n-1}u_n$, we define its \emph{generation} as $| u |:=n$ and its \emph{parent particle} as $\overleftarrow{u}:=u_1\dots u_{n-1}$.
	For any particle $u \in \mathcal U$ and any subset $\mathbf a \subset \mathcal U$, we define the \emph{number of children of $u$ in $\mathbf a$} as $l_u(\mathbf a) := \#\{\alpha\in \mathbf a:\overleftarrow{\alpha}=u\} $.
	We also define the \emph{height} of $\mathbf a$ as $|\mathbf a|:=\sup_{\alpha\in \mathbf a}|\alpha|$ and its \emph{population in the $n$th generation} as $X_n(\mathbf a):=\#\{u\in \mathbf a:|u|=n\}$.
	A \emph{tree} $ \mathbf t $ is defined as a subset of $\mathcal U$ such that there exists an $\mathbb N_0$-valued sequence $(l_u)_{u\in \mathcal U}$,
	indexed by $\mathcal U$, satisfying
\[
		 \mathbf t
	=\{u_1\dots u_m\in \mathcal U: m\ge 0, u_j\leq l_{u_1\dots u_{j-1}}, \forall  j=1,\dots,m\}.
\]
	A \emph{spine} $ \mathbf v$ on a  tree $ \mathbf t $ is defined as a sequence of particles $\{v^{(k)}:k=0,1,\dots,| \mathbf t |\}\subset \mathbf t $ such that $v^{(0)}=\emptyset$ and $\overleftarrow{v^{(k)}}=v^{(k-1)}$ for any $k=1,\dots, | \mathbf t |$.
	In the case that $| \mathbf t |=\infty$, we simply write $k=0,1,\dots$ as $k=0,1,\dots, | \mathbf t |$.

	Fix a generation number $n\in \mathbb N$. Define the following spaces:
\begin{itemize}
\item
	\emph{The space of trees with height no more than $n$},
\[
		\mathbb T_{\leq n}
	:=\{ \mathbf t : \mathbf t \text{ is a tree with }| \mathbf t | \leq n\}.
\]
\item
	\emph {The space of $n$-height trees with one distinguishable spine},
\[
		\dot{\mathbb T}_n
	:=\{( \mathbf t , \mathbf v): \mathbf t  \text{ is a tree with } | \mathbf t |=n,  \mathbf v \text{ is a spine on }  \mathbf t \}.
\]
\item
	\emph{The space of $n$-height trees with two different distinguishable spines},
\[
		\ddot{\mathbb T}_n
	:=\{( \mathbf t , \mathbf v, \mathbf v'):( \mathbf t , \mathbf v)\in\dot{\mathbb T}_n,( \mathbf t , \mathbf v')\in\dot{\mathbb T}_n, \mathbf v\neq \mathbf v'\}.
\]
\end{itemize}

	Let $(L_u)_{u\in\mathcal U}$ be a collection of independent random variables with law $\mu$, indexed by $\mathcal U$.
	Denote by $T$ the random tree defined by
\[
		T
	:=\{u_1\dots u_m\in \mathcal U: 0\le m\le n, u_j\leq L_{u_1\dots u_{j-1}},\forall j=1,\dots,m\}.
\]
	We refer to $T$ as a \emph{$\mu$-Galton-Watson tree with height no more than n} since its population $(X_m(T))_{0\le m\le n}$ is a $\mu$-Galton-Watson process stopped at generation $n$.
	Define the \emph{$\mu$-Galton-Watson measure $\mathbf G_n$} on $\mathbb T_{\leq n}$ as the law of the random tree $T$.
	That is, for any $ \mathbf t \in\mathbb T_{\leq n}$,
\[
		\mathbf G_n( \mathbf t )
    :=P(T= \mathbf t )
	=P(L_u=l_u( \mathbf t )\text{ for any } u\in \mathbf t  \text{ with }|u|<n)
	=\prod_{u\in  \mathbf t :|u|<n}\mu(l_u( \mathbf t )).
\]

	Recall that $\dot L$ is an $L$-transform of $L$.
	Define $\dot C$ as a random number which, conditioned on $\dot L$, is uniformly distributed on $\{1,\dots,\dot L\}$.
	Independent of $(L_u)_{u\in\mathcal U}$, let $(\dot L_u,\dot C_u)_{u\in \mathcal U}$ be a collection of independent copies of $(\dot L,\dot C)$, indexed by $\mathcal U$.
	We then use $(L_u)_{u\in\mathcal U}$ and $(\dot L_u,\dot C_u)_{u\in\mathcal U}$ as the building blocks to construct the size-biased $\mu$-Galton-Watson tree $\dot T$ and its distinguishable spine $\dot V$ following the steps described in Section \ref{sec:model}.
	We use $L_u$ as the number of children of particle $u$ if $u$ is unmarked and use $\dot L_u$ if $u$ is marked.
	In the latter case, we always set the
	$\dot C_u$-th child of $u$, i.e. particle $u \dot C_u$,
	as the new marked particle.
	For convenience, we stop the system at generation $n$. To be precise, the random spine $\dot V$ is defined by
\[
		\dot V
	:=\{v_1\dots v_m\in \mathcal U:0\le m\le n, v_j=\dot C_{v_1\dots v_{j-1}},\forall j=1,\dots,m\},
\]
	and the random tree $\dot T$ is defined by
\[
		\dot T
	:=\{u_1\dots u_m\in\mathcal U: 0\le m\le n,u_j\leq \tilde L_{u_1\dots u_{j-1}},\forall j=1,\dots,m\},
\]
	where, for any $u\in\mathcal U$, $\tilde L_u:=L_u\mathbf 1_{u\not\in \dot V}+\dot L_u\mathbf 1_{u\in \dot V}$.

	We now consider the distribution of the $\dot{\mathbb T}_n$-valued random element $(\dot T,\dot V)$.
	For any $( \mathbf t , \mathbf v)\in\dot{\mathbb T}_n$, the event $\{(\dot T,\dot V)=( \mathbf t , \mathbf v)\}$ occurs if and only if:
\begin{itemize}
\item
    $L_u=l_u( \mathbf t )$ for each $u\in  \mathbf t \setminus \mathbf v$ with $| u |<n$ and
\item
	$(\dot L_{v_1\dots v_m},\dot C_{v_1\dots v_m})=(l_{v_1\dots v_m}( \mathbf t ),v_{m+1})$ for each $v_1\dots v_{m+1}\in \mathbf v$ with $0\le m\le n-1$.
\end{itemize}
    Therefore, the distribution of $(\dot T,\dot V)$ can be determined by
\begin{equation}
\label{eq:treespinemeasure}
		P((\dot T,\dot V)=( \mathbf t , \mathbf v))
	=\prod_{u\in  \mathbf t \setminus \mathbf v:|u|<n}\mu(l_u( \mathbf t ))
	\cdot \prod_{u\in  \mathbf v:| u| <n}l_u( \mathbf t )\mu(l_u( \mathbf t ))\frac{1}{l_u( \mathbf t )}
	= \mathbf G_n( \mathbf t ).
\end{equation}
	
	The \emph{size-biased $\mu$-Galton-Watson measure $\dot {\mathbf G}_n$} on $\mathbb T_{\leq n}$ is then defined as the law of the $\mathbb T_{\leq n}$-valued random element $\dot T$. That is, for any $ \mathbf t \in\mathbb T_{\leq n}$,
\begin{equation}
\label{eq:sizebiasedGWmeasure}
\begin{split}
		\dot {\mathbf G}_n( \mathbf t )
	&:= P(\dot T= \mathbf t )
	= \sum_{ \mathbf v:( \mathbf t , \mathbf v)\in \dot{\mathbb T}_n} P((\dot T,\dot V)=( \mathbf t , \mathbf v))
	\\&= \#\{ \mathbf v:( \mathbf t , \mathbf v)\in \dot{\mathbb T}_n\} \cdot \mathbf G_n( \mathbf t )
	= X_n( \mathbf t ) \cdot \mathbf G_n( \mathbf t ).
\end{split}
\end{equation}

	Equations \eqref{eq:treespinemeasure}, \eqref{eq:sizebiasedGWmeasure} and their consequence \eqref{eq:htransformation} were first obtained in \cite{lyons1995conceptual}.
	We use these equations to help us to understand how the $k(k-1)$-type size-biased $\mu$-Galton-Watson tree can be represented.
	
	Recall that $K_n$ is a random generation number uniformly distributed on $\{0,\dots,n-1\}$,
	and $\ddot L$ is an $L(L-1)$-transform of $L$.
	Define $(\ddot C,\ddot C')$ as a random vector which, conditioned on $\ddot L$, is uniformly distributed on $\{(i,j)\in\mathbb N^2:1\leq i\neq j\leq \ddot L\}$.
	Suppose that $(L_u)_{u\in\mathcal U}, (\dot L_u,\dot C_u)_{u\in \mathcal U}$, $(\ddot L,\ddot C,\ddot C')$ and $K_n$ are independent of each other.
	We now use these elements to build the $k(k-1)$-type size-biased $\mu$-Galton-Watson tree $\ddot T$ and its two different distinguishable spines $\ddot V$ and $\ddot V'$ following the steps described in Section \ref{sec:model}.
	Write $C_u:=\dot C_u\mathbf 1_{|u|\neq K_n}+\ddot C\mathbf 1_{|u|=K_n}$ and $C'_u:=\dot C_u\mathbf 1_{|u|\neq K_n}+\ddot C'\mathbf 1_{|u|=K_n}$.
	We define the random spines $\ddot V$ and $\ddot V'$ as
\[ \begin{split}
        \ddot V
	&:= \{v_1\dots v_m\in \mathcal U:0\le m\le n, v_j= C_{v_1\dots v_{j-1}},\forall j=1,\dots,m\},
	\\ \ddot V'
	&:= \{v_1\dots v_m\in \mathcal U:0\le m \le n, v_j= C'_{v_1\dots v_{j-1}},\forall j=1,\dots,m\},
\end{split}\]
	and the random tree $\ddot T$ as
\[
	    \ddot T
	:=
		\{u_1\dots u_m\in\mathcal U: 0\le m\le n,u_j\leq L''_{u_1\dots u_{j-1}},\forall j=1,\dots,m\},
\]
	where, for any $u\in\mathcal U$, $L''_u:=L_u \mathbf 1_{u\not\in \ddot V\cup\ddot V'}+\dot L_u \mathbf 1_{u\in \ddot V\cup\ddot V',|u|\neq K_n}+\ddot L\mathbf 1_{u\in \ddot V\cup\ddot V',|u|=K_n}$.

	We now consider the distribution of $(\ddot T,\ddot V,\ddot V')$.
	For any $( \mathbf t , \mathbf v, \mathbf v')\in\ddot {\mathbb T}_n$, the event $\{(\ddot T,\ddot V,\ddot V')=( \mathbf t , \mathbf v, \mathbf v')\}$ occurs if and only if:
\begin{itemize}
\item
    $K_n=k_n:=| \mathbf v\cap \mathbf v'|$,
\item
    $L_u=l_u( \mathbf t )$ for each $u\in  \mathbf t \setminus( \mathbf v\cup \mathbf v')$ with $| u| <n$,
\item
	$(\dot L_{v_1\dots v_m},\dot C_{v_1\dots v_m})=(l_{v_1\dots v_m}( \mathbf t ),v_{m+1})$ for each $v_1\dots v_mv_{m+1}\in \mathbf v\cup \mathbf v'$ with $k_n\neq m<n$ and
\item
	$(\ddot L,\ddot C,\ddot C')=(l_{v_1\dots v_{k_n}}( \mathbf t ),v_{k_n+1},v'_{k_n+1})$ for $v_1\dots v_{k_n}v_{k_n+1}\in \mathbf v$ and $v_1\dots v_{k_n}v'_{k_n+1}\in \mathbf v'$.
\end{itemize}
	Using this analysis, we get that
\[\begin{split}
		P\big((\ddot T,\ddot V,\ddot V')=( \mathbf t , \mathbf v, \mathbf v')\big)
		&=\frac{1}{n} \cdot \prod_{u\in  \mathbf t \setminus( \mathbf v\cup  \mathbf v'):|u|<n} \mu(l_u( \mathbf t )) \cdot \prod_{u\in  \mathbf v\cup  \mathbf v':k_n\neq|u|<n}l_u( \mathbf t ) \mu(l_u( \mathbf t ))\frac{1}{l_u( \mathbf t )}
    \\&\qquad \cdot \prod_{u\in  \mathbf v \cup  \mathbf v':|u|=k_n}\frac{l_u( \mathbf t )(l_u( \mathbf t )-1) \mu(l_u( \mathbf t ))}{\sigma^2}\frac{1}{l_u( \mathbf t )(l_u( \mathbf t )-1)}\\
	&= \frac{1}{n\sigma^2} \mathbf G_n( \mathbf t ).
\end{split}\]
	
	The \emph{$k(k-1)$-type size-biased $\mu$-Galton-Watson measure $\ddot{\mathbf G}_n$} on $\mathbb T_{\leq n}$ is then defined as the law of the random element $\ddot T$. That is, for any $ \mathbf t \in\mathbb T_{\leq n}$,
\begin{equation}
\label{eq:k(k-1)typesizebiasedGWmeasure}
\begin{split}
		\ddot{\mathbf G}_n( \mathbf t )
	&:= P(\ddot T= \mathbf t )
	= \sum_{( \mathbf v, \mathbf v'):( \mathbf t , \mathbf v, \mathbf v')\in \ddot {\mathbb T}_n} P\big((\ddot T,\ddot V,\ddot V')=( \mathbf t , \mathbf v, \mathbf v')\big)
	\\&= \#\{( \mathbf v, \mathbf v'):( \mathbf t , \mathbf v, \mathbf v')\in \ddot {\mathbb T}_n\} \cdot \frac{\mathbf G_n( \mathbf t )}{n\sigma^2}
	= \frac{X_n( \mathbf t )(X_n( \mathbf t )-1)}{n\sigma^2} \cdot{\mathbf G}_n( \mathbf t ).
\end{split}
\end{equation}

	We note in passing that, because of the way they are constructed, the measures $(\ddot{\mathbf G}_n)_{n\ge 1}$ are not consistent, that is, the measure $\ddot{\mathbf G}_n$ is not the restriction of $\ddot{\mathbf G}_{n+1}$.
	This implies that the change of measure in Theorem \ref{thm: change of measure} is not a martingale change of measure.
\medskip
\begin{proof}[Proof of Theorem \ref{thm: change of measure}]
    Note that 
    \[
    \{(X_m( \mathbf t ))_{0\le m\le n}; {\mathbf G}_n\}  \overset{d}{=} (Z_m)_{0\le m\le n}
\quad    \mbox{and} \quad \{(X_m( \mathbf t ))_{0\le m\le n};\ddot{\mathbf G}_n\}  \overset{d}{=} (\ddot Z_m)_{0\le m\le n}.
    \]
    According to \eqref{eq:k(k-1)typesizebiasedGWmeasure}, for any bounded Borel function $g$ on $\mathbb N_0^n$, we can verify that
\begin{equation} \label{eq:proofofchangeofmeasure}
\begin{split}
	&E [ g ( \ddot Z_1^{(n)}, \dots, \ddot Z_n^{(n)})]
	= \ddot{\mathbf G}_n [g ( X_1(  \mathbf t ), \dots, X_n(  \mathbf t ))]
    \\ &\quad = {\mathbf G}_n \big[ \frac { X_n( \mathbf t ) ( X_n( \mathbf t ) - 1)} {n \sigma^2} g (X_1( \mathbf t ), \dots, X_n( \mathbf t ))\big]
	\\&\quad = \frac { 1} { n \sigma^2} E[ Z_n ( Z_n - 1) g( Z_1, \dots, Z_n)].
\end{split}
\end{equation}
	Taking $g\equiv 1$ in equation \eqref{eq:proofofchangeofmeasure}, we get that
\begin{equation}
	\label{eq: second moment}
	E [Z_n(Z_n-1)]= E [\dot Z_n - 1]= n\sigma^2.
\end{equation}
	\end{proof}

\subsection{Spine decompositions.}
\label{sec:spinesdecomposition}

	Using the	notation
	introduced in the previous	section, we are now ready to
	give a precise meaning to \eqref{eq: Our insight}:
\begin{prop}\label{prop: size-biased add-on of size-biased tree }
	Let $(\dot Z_m)_{0 \leq m \leq n}$ be the population of a size-biased $\mu$-Galton watson tree and $(\ddot Z^{(n)}_m)_{0 \leq m \leq n}$ be the population of a $k(k-1)$-type size-biased $\mu$-Galton-Watson tree with height $n$.
	Suppose that $\mu$ satisfies \eqref{eq:mean} and \eqref{eq:variance}.
	Then, we have
\[
	E [ e^{- \lambda \ddot Z_n^{(n)}} ]
	= E [e^{-\lambda \dot Z_n}] E[g(\lambda, \floor{Un})e^{-\lambda \dot Z_{\floor{Un}}}],
\]
where $U$ is a uniform random variable on $[0,1]$ independent of $\{\dot Z_m: 0\le m\le n\}$;
and $g(\lambda, m)$ is a function on $[0,\infty) \times \mathbb N_0$ such that
$g(\lambda, m) \to 1$, uniformly in $\lambda$ as $m\to \infty$.
\end{prop}

\begin{proof}
	For any particle $u=u_1\dots u_n$, we define
$	[\emptyset, u]
	:= \{u_1\dots u_j:j=0,\dots, n \}$
	as the \emph{descending family line from $\emptyset$ to $u$}.
	The particles in $\dot T$ can be separated according to their nearest spine ancestor.
	For each $k = 0, \dots, n$, we write
$\dot A_k
	:= \{u\in\dot T:| [\emptyset, u] \cap \dot V |=k\}.$
	Then
\begin{equation}
\label{eq:generationseperation}
		X_n(\dot T)
	=
		\sum_{k=0}^nX_n(\dot A_k).
\end{equation}
	Notice that the right side of the above equation is a 
	sum of independent random variables;
	and from their construction, we see that $X_n(\dot A_k) \overset{d}= Z_{n-k-1}^{(\dot L - 1)}$.
	Here,  $Z^{(\dot L - 1)}_{(-1)}:= 1$ and $(Z^{(\dot L - 1)}_m)_{m\in \mathbb N_0}$ denotes a $\mu$-Galton-Watson process with $Z_0^{(\dot L - 1)}$ distributed according to $\dot L - 1$.
	Taking	Laplace transforms
	on both sides of \eqref{eq:generationseperation} we get
\begin{equation} \label{eq: laplace transform of one-spine decomposition}
	E [e^{-\lambda \dot Z_n}]
	= \prod_{k = 0}^n E[ e^{-\lambda Z^{(\dot L - 1)}_{n-k-1}} ].
\end{equation}
	
	Similarly, we consider the $k(k-1)$-type size-biased $\mu$-Galton-Watson tree $(\ddot T,\ddot V,\ddot V')$.
	Write
\[
	\ddot A^l_k := \{u\in\ddot T: | [\emptyset, u]\cap \ddot V | = k, [\emptyset , u] \cap (\ddot V' \setminus \ddot V ) = \emptyset\}
\]
	and
\[
	\ddot A^s_k := \{u\in\ddot T: | [\emptyset, u]\cap \ddot V' | = k, [\emptyset , u] \cap (\ddot V' \setminus \ddot V) \neq \emptyset\}.
\]
	Then, 
\begin{equation}\label{eq:rawtwospinedecomposition}
		X_n(\ddot T)
	=
		\sum_{k=0}^nX_n(\ddot A^l_k) + \sum_{k=K_n + 1}^n X_n(\ddot A^s_k).
\end{equation}
	Notice that, conditioning on $K_n = m$ with $m\in\{0,\dots,n-1\}$, the right side of the above equation is a 
   sum of independent random variables; and from their construction, we see that
	$X_n(\ddot A^l_k) \overset{d}{=} Z^{(\dot L - 1)}_{n-k-1}$
	for each $k \neq m$;
	$X_n(\ddot A^l_m) \overset{d}{=} Z^{(\ddot L - 2)}_{n-m-1}$;
	and $X_n(\ddot A^s_k) \overset{d}{=} Z^{(\dot L - 1)}_{n-k-1}$ for each $k \geq m+1$.
	Here, $Z^{(\ddot L - 2)}_{(-1)}:= 1$ and $(Z^{(\ddot L - 2)}_k)_{k\in \mathbb N_0}$ is a $\mu$-Galton-Watson process with initial population distributed according to $\ddot L-2$.

	Taking Laplace transform on both sides of \eqref{eq:rawtwospinedecomposition}  and using \eqref{eq: laplace transform of one-spine decomposition}, we get
\[  \begin{split}
	&E [ e^{- \lambda \ddot Z_n^{(n)}} ]
	= \frac{1}{n}\sum_{m=0}^{n-1} \Big( \prod_{k=0,k\neq m}^{n} E[e^{-\lambda Z^{(\dot L - 1)}_{n-k-1}}] \Big) \cdot E [e^{-\lambda Z^{(\ddot L - 2)}_{n-m-1}}] \cdot \Big(\prod_{k= m+1}^n E [e^{-\lambda Z^{(\dot L - 1)}_{n-k-1}}]\Big)
	 \\&\quad = E [e^{-\lambda \dot Z_n}]  \frac{1}{n} \sum_{m=0}^{n-1}   \frac{ E [e^{-\lambda Z^{(\ddot L - 2)}_{n-m-1} }] }{ E[e^{-\lambda Z^{(\dot L - 1)}_{n-m-1}  }] } \cdot E[e^{- \lambda \dot Z_{n-m-1}}]
	 \\&\quad = E [e^{-\lambda \dot Z_n}]  \frac{1}{n}\sum_{m=0}^{n-1} \frac{ E [e^{-\lambda Z^{(\ddot L - 2)}_{m}}] }{ E[e^{-\lambda Z^{(\dot L - 1)}_{m}}] } \cdot E[e^{- \lambda \dot Z_{m}}]
	= E [e^{-\lambda \dot Z_n}] E[g(\lambda,\floor{Un})e^{-\lambda \dot Z_{\floor{Un}}}],
\end{split}
\]
	where
\[
	P( Z^{(\ddot L - 2)}_m=0 )
	\leq	g(\lambda,m)
	: = \frac{ E [e^{-\lambda Z^{(\ddot L - 2)}_{m}}] }{ E[e^{-\lambda Z^{(\dot L - 1)}_{m}}] }
	\leq P ( Z^{(\dot L - 1)}_m = 0 )^{-1}.
\]
	Notice that, from the criticality, $P (  Z^{(\ddot L - 2)}_m=0 )$ and $P ( Z^{(\dot L - 1)}_m = 0 )^{-1}$ converge to $1$.
\end{proof}

\section{Proofs}
\label{sec: proofs}
\begin{proof}[Proof of Theorem \ref{thm: Kolmogrov and Yaglom theorem}(\ref{thm:kolmogorov})]
	Denote by $B_n^j$ the event that the Galton-Watson process 
    $(Z_n)_{n\geq 0}$ 
survives up to generation $n$, and the left-most particle in the $n$-th generation is a descendant of $j$th particle of the first generation.
	Write $q_n = P[Z_n = 0] = f^{(n)}(0)$ and $p_n = 1- q_n$ where $f$ is the probability generating function of the offspring distribution $\mu$.
	Then
\begin{equation}
\label{eq: iterate of conditional expectation}
\begin{split}
	&E[Z_n| Z_n>0]
	= \sum_{k=1}^\infty E[Z_n; Z_1=k| Z_n>0]
	= p_n^{-1} \sum_{k=1}^\infty E[Z_n; Z_1=k;Z_{n}>0]
	\\&\quad = p_n^{-1} \sum_{k=1}^\infty \sum_{j=1}^k E[Z_n; Z_1=k;B_n^j]
	= p_n^{-1} \sum_{k=1}^\infty \sum_{j=1}^k P[Z_1=k;B_n^j] E[Z_n| Z_1=k,B_n^j]
	\\&\quad = p_n^{-1} \sum_{k=1}^\infty \sum_{j=1}^k P[Z_1=k;B_n^j] \Big( E[Z_{n-1}| Z_{n-1}>0] +k-j\Big)
	\\&\quad = E[Z_{n-1}|Z_{n-1}> 0]  + \frac{p_{n-1}}{p_n}\sum_{k=1}^\infty \sum_{j=1}^k \mu(k) q_{n-1}^{j-1}(k-j).
\end{split}
\end{equation}
	The criticality implies that $q_n \uparrow 1$ while $n \to \infty$, and that
\[
	\frac{p_n}{p_{n-1}} = \frac{1- f^{(n)}(0)}{1-f^{(n-1)}(0)} = \frac{1- f(q_{n-1})}{1-q_{n-1}} \xrightarrow[n \to \infty]{} f'(1) = 1.
\]
	By the monotone convergence theorem,
\[
	\frac{p_{n-1}}{p_n} \sum_{k=1}^\infty \sum_{j=1}^k \mu(k) q_{n-1}^{j-1} (k-j)
	\xrightarrow[n \to \infty]{} \sum_{k=1}^\infty \sum_{j=1}^k \mu(k) (k-j)
	=  \sum_{k=1}^\infty \mu(k) k(k-1)/2
	= \frac{\sigma^2}{2}.
\]
	Now combining \eqref{eq: iterate of conditional expectation} with above, we get
\[\begin{split}
	&\frac{1}{n P(Z_n > 0)}
	= \frac{1}{n} E[Z_n | Z_n > 0]
	= \frac{1}{n}E[Z_0| Z_0 > 0] + \frac{1}{n} \sum_{m = 1}^n \frac{p_{m-1}}{p_m}\sum_{k=1}^\infty \sum_{j=1}^k \mu(k) q_{m-1}^{j-1}(k-j)
	\\&\quad \xrightarrow[n \to \infty]{} \frac{\sigma^2}{2}.
\end{split}\]
\end{proof}

	In order to compare distributions using their size-biased add-on structures, we need the following lemma:
\begin{lem}\label{lem: compare}
	Let $X_0$ and $X_1$ be two non-negative random variables with the same mean $a = E[X_0] = E[X_1] \in (0,\infty)$.
	Let $F_0$ be defined by $E[e^{-\lambda \dot X_0}] = E[e^{-\lambda X_0}] F_0(\lambda)$,
	where $\dot X_0$ is an $X_0$-transform of $X_0$,
	and $F_1$ be defined by $E[e^{-\lambda \dot X_1}] = E[e^{-\lambda X_1}] F_1(\lambda)$,
	where $\dot X_1$ is an $X_1$-transform of $X_1$.
	Then,
	\[
  \big| E[e^{-\lambda X_0}] - E[e^{-\lambda X_1}] \big|
\leq a \int_0^\lambda| F_0(s) - F_1(s) |ds, \quad \lambda \geq 0.
	\]
\end{lem}
\begin{proof}
	Since $\dot X_0$ is an $X_0$-transform of $X_0$,
	we have
	\[
	\partial_\lambda ( -\ln E[e^{-\lambda X_0}]) = \frac { E[X_0 e^{-\lambda X_0}]}{ E[e^{-\lambda X_0}] }
	= \frac{a E[e^{-\lambda \dot X_0}]}{E[e^{-\lambda X_0}]}
	= a F_0(\lambda).
	\]
	Similarly, $\partial_\lambda ( -\ln E[e^{-\lambda X_1}]) = a F_1(\lambda)$.
	Therefore,
	since $x - \ln x$ is decreasing on $[0,1]$,
	\[\begin{split}
    \big| E[e^{-\lambda X_0}] - E[e^{-\lambda X_1}] \big|
	&\leq \big| \ln E[e^{-\lambda X_0}] - \ln E[e^{-\lambda X_1}] \big|
	= a\big| \int_0^\lambda F_0(s)ds - \int_0^\lambda F_1(s)ds \big|
	\\& \leq a \int_0^\lambda| F_0(s) - F_1(s) |ds
	\end{split}\]
	as desired.
\end{proof}

   We are now ready to prove Lemma \ref{lem: our equation}.
	It is elementary to verify that if $Y$ is exponentially distributed, then it satisfies \eqref{eq: x2 type size-biased equation for exponential distribution}.
	So we 
	only need to show
	that if $Y$ is a strictly positive random variable with finite second moment, then \eqref{eq: x2 type size-biased equation for exponential distribution} implies that it is exponentially distributed.
	The following lemma will be used to prove this.

\begin{lem}\label{lem: zero inequality}
	Suppose that $c>0$ is a constant, and $F$  is a non-negative bounded function on $[0,\infty)$ satisfying that, for any $\lambda\geq 0$,
\begin{equation}\label{eq: zero inequality}
	F(\lambda)
	\leq
	\frac{1}{c}\int_0^1du
	\int_0^\lambda F(us)ds.
\end{equation}
	Then $F\equiv 0$.
\end{lem}

\begin{proof}
	By dividing both sides of \eqref{eq: zero inequality} by $\|F\|_\infty$, without loss of any generality, we can assume $F$ is bounded by $1$.
    We prove this lemma by  contradiction.
	Assume that
\begin{equation}
\label{eq: contradiction}
	\rho
	:= \inf\{x \geq 0: F(x) \neq 0\}
	< \infty,
\end{equation}
with the convention $\inf \emptyset=\infty$.
	Then, for each $\lambda > 0$,
\[
	F(\rho + \lambda)
	= \frac{1}{c} \int_0^1 du \int_0^{\rho +\lambda} F(us) ds
	= \frac{1}{c} \int_0^1 du \int_\rho^{\rho+\lambda} F(us) ds \leq \frac{\lambda}{c}.
\]
	Using this new upper bound, we have
\[
	F(\rho + \lambda)
	= \frac{1}{c} \int_0^1 du \int_\rho^{\rho+\lambda} F(us) ds
	\leq \frac{1}{c} \int_0^1 du \int_\rho^{\rho+\lambda} \frac{\lambda}{c}ds
	\leq \frac{\lambda^2}{c^2}.
\]
	Repeating this process, we have $F(\rho + \lambda) \leq \frac{\lambda^m}{c^m}$ for each $m \in \mathbb N$, which implies that $F = 0$ on $[\rho, \rho + c)$.
	This, however, contradicts \eqref{eq: contradiction}.
\end{proof}

\begin{proof}[Proof of Lemma \ref{lem: our equation}]
	Suppose that $Y$ is a strictly positive random variable with finite second moment, and \eqref{eq: x2 type size-biased equation for exponential distribution} is true.
	Define
$
	a
	:= E[\dot Y] \in (0,\infty)
$.
 	Consider an exponential random variable $\mathbf e$ with mean $a/2$.
	It is elementary to verify that
	$\mathbf e$ satisfies
	\eqref{eq: x2 type size-biased equation for exponential distribution}, in the sense that
$
	\ddot {\mathbf e} \overset{d} = \dot {\mathbf e}+U\dot {\mathbf e}',
$
	where $\dot {\mathbf e}$ and $\dot {\mathbf e}'$ are
both $\mathbf e$-transforms of $\mathbf e$,
	$\ddot {\mathbf e}$ is an $\mathbf e^2$-transform of $\mathbf e$,
	$U$ is a uniform random variable on $[0,1]$, and $\dot {\mathbf e}$, $\dot {\mathbf e}'$,
	$\ddot {\mathbf e}$ and $U$ are independent.
	Notice that $E[\dot {\mathbf e}] = a$, therefore we can compare the distribution of $\dot Y$ 
	with that of $\dot {\mathbf e}$
	using Lemma \ref{lem: compare}.
	This gives that
\[
	\big|E[ e^{-\lambda \dot Y}] - E[ e^{-\lambda \dot {\mathbf e}}] \big|
	\leq  a  \int_0^\lambda \int_0^1 \big| E [e^{-s u \dot Y}] - E [e^{-s u \dot {\mathbf e}}] \big| du ds,
	\quad \lambda \geq 0,
\]
	which, according to Lemma \ref{lem: zero inequality}, says that
	$\dot Y \overset{d} = \dot {\mathbf e}$.
	Since $Y$ and $\mathbf e$ are strictly positive, according to \eqref{eq: conditional and size-biased transform}, we have
\[
	E[1-e^{-\lambda Y}]/ E[Y] = E[1-e^{-\lambda \mathbf e}]/ E[\mathbf e], \quad \lambda \geq 0.
\]
	Letting $\lambda \to \infty$, we get $E[Y] = E[\mathbf e]$. Therefore, $Y \overset{d} = \mathbf e$ as desired.
\end{proof}

\begin{proof}[Proof of Theorem \ref{thm: Kolmogrov and Yaglom theorem}\eqref{thm:yaglom}.]
	Consider an exponential random variable $Y$ with mean $\sigma^2/2$.
	Let $\dot Y$ be a $Y$-transform of $Y$.
	As in Section \ref{sec: Methods}, 
    we only need to prove that
    $\dot Z_n/n$ converge weakly to $\dot Y$.
	From Proposition \ref{prop: size-biased add-on of size-biased tree }, we know that
	\[
	E [ e^{- \lambda \ddot Z_n^{(n)}} ]
		= E [e^{-\lambda \dot Z_n}] E[g(\lambda, \floor{Un})e^{-\lambda \dot Z_{\floor{Un}}}],
	\]
	where $U$ is a uniform  random variable on $[0,1]$ independent of $\{\dot Z_m: 0\le m\le n\}$;
	and $g(\lambda, m)$ is a function on $[0,\infty) \times \mathbb N_0$ such that
	$g(\lambda, m) \to 1$, uniformly in $\lambda$ as $m\to \infty$.
	After a renormalization, we have that
\begin{equation}\label{eq: renormalized size-biased add-on equation}
	E [ e^{- \lambda \frac{\ddot Z_n^{(n)}-1}{n}} ]
	= E [e^{-\lambda \frac{\dot Z_n - 1}{n}}] E\big[g\big(\frac{\lambda}{n}, \floor{Un} \big)e^{-\lambda U \frac{\dot Z_{\floor{Un}}}{Un} }\big],
	\quad \lambda \geq 0.
\end{equation}
	According to Theorem \ref{thm: change of measure}, one can verify that
	$(\ddot Z_n^{(n)} - 1)/n$ is a $(\dot Z_n - 1)/n$ transform of $(\dot Z_n - 1)/n$.
	Therefore, the above equation can be viewed as the size-biased add-on structure for the random variable $(\dot Z_n - 1)/n$.
	It is easy to see that the mean of $\dot Y$ is $\sigma^2$.
	According to \eqref{eq: second moment}, the mean of $(\dot Z_n - 1)/n$ is also $\sigma^2$.
Then comparing the distribution of $(\dot Z_n - 1)/n$ with that of $\dot Y$, and  using Lemma \ref{lem: compare}, we get that
\[
	\big| E[e^{-\lambda \frac{\dot Z_n - 1}{n}}] - E[e^{-\lambda \dot Y}]\big|
	 \leq \sigma^2 \int_0^\lambda ds \int_0^1 \big| g(\frac{s}{n}, \floor{un}) E[e^{-su \frac { \dot Z_{\floor{un}} } {un} }] - E[e^{- su \dot Y}]\big| du.
\]
Taking $n\to \infty$ and using the reverse Fatou's lemma, we arrive at
\[
	M(\lambda)
	\leq \sigma^2 \int_0^1du \int_0^\lambda M(us)ds,
	\quad \lambda\geq 0,
\]
	where
$M(\lambda) := \limsup_{n\to \infty} | E[ e^{- \lambda \frac{\dot Z_n }{n}}] - E[e^{-\lambda \dot Y}]|$.
	Thus by Lemma \ref{lem: zero inequality}, we have $M \equiv 0$, which says that $\dot Z_n/n$ converges weakly to $\dot Y$.
\end{proof}

\noindent
{\bf Acknowledgements:}  We thank the referee for helpful comments on the first version of this paper.

\end{document}